\documentclass[11pt]{amsart}
\usepackage[T1]{fontenc}
\usepackage{lmodern}
\usepackage{ifthen}
\usepackage{amsfonts}
\usepackage{amsxtra}
\usepackage{amssymb}
\usepackage{amsthm}
\usepackage{array}
\usepackage[margin=1in]{geometry}
\usepackage{xcolor}
\definecolor{cite}{rgb}{0.50,0.00,1.00}
\definecolor{url}{rgb}{0.00,0.50,0.75}
\definecolor{link}{rgb}{0.00,0.00,0.50}
\usepackage[colorlinks,linkcolor=link,urlcolor=url,citecolor=cite,pagebackref,breaklinks]{hyperref}
\usepackage{mathtools}
\usepackage{mathrsfs}
\usepackage[all]{xy}
\usepackage[lite,abbrev,msc-links]{amsrefs}

\makeindex

\theoremstyle{definition} %标题与编号为黑体, 正文为正常字体
\newtheorem{Unity}{Unity}[section] %\newtheorem{定理环境名}{标题}[主计数器名]
\newtheorem*{defn*}{Definition} %\newtheorem*{定理环境名}[已定义定理环境名]{标题} 手动编号, 不自动编号
\newtheorem{defn}[Unity]{Definition} %\newtheorem{定理环境名}[已定义定理环境名]{标题} 与当前环境共用同一个序号计数器

\theoremstyle{plain} %标题与编号为黑体, 正文为斜体
\newtheorem*{thm*}{Theorem}
\newtheorem{thm}[Unity]{Theorem}
\newtheorem{prop}[Unity]{Proposition}
\newtheorem*{cor*}{Corollary}
\newtheorem{cor}[Unity]{Corollary}
\newtheorem{lem}[Unity]{Lemma}

\theoremstyle{remark} %标题与编号为斜体, 正文为正常字体
\newtheorem*{rmk*}{Remark}
\newtheorem{rmk}[Unity]{Remark}

\numberwithin{Unity}{section}%\numberwithin{计数器}{主计数器}

\newcommand{\sF}{\mathscr{F}}
\newcommand{\sG}{\mathscr{G}}

\newcommand{\sM}{\mathscr{M}}

\newcommand{\sO}{\mathscr{O}}
\newcommand{\sP}{\mathscr{P}}

\newcommand{\sX}{\mathscr{X}}

\newcommand{\bZ}{\mathbb{Z}}
\newcommand{\bN}{\mathbb{N}}

\newcommand{\spec}{\textrm{Spec}}
\begin{document}
\title{Smoothing of semistable Fano varieties}
\author{Junchao Shentu}%\thanks{1 The first author is supported by the ...[grant numbers ...].}
\address{Shanghai Center for Mathematical Science, Fudan University, Shanghai, P. R. China}
\email{stjc@amss.ac.cn}
\begin{abstract}
Given any field $k$ (not necessarily perfect), we study the smoothing of a semistable Fano variety over $k$. In characteristic 0, the reduced semistable Fano degenerate fibers of Mori fibrations are classified. In positive characteristic, under a suitable $W_2$ lifting assumption, we prove that a semistable Fano variety always appears as a degenerate fiber in a semistable family if it has a global log structure (in the sense of Fontaine-Illusie-Kato) of semistable type. A similar smoothing result over a mixed characteristic base is also obtained.
\end{abstract}
\maketitle
\section{Introduction}
In this paper we investigate the possible reduced semistable degenerate fibers of a smooth Fano family. The investigation is motivated by the following two problems:
\begin{enumerate}
  \item Given a projective algebraic variety $X$ (defined over the complex number field) whose Kodaira dimension is $-\infty$. Then conjecturally there is a modification $X'\rightarrow X$ and a fibration $f:X'\rightarrow B$ (Mori fibration) such that $X'$ is smooth projective and the general fibers of $f$ are smooth Fano varieties. Resembling to Kodaira's theory of degenerations of elliptic curves, the geometry of the singular fibers (usually non-normal) of $f$ have great influence on the birational geometry of $X$. The simplest possible non-normal singularities in the fibers of $f$ are the semistable singularities (those which are analytically isomorphic to a product of normal crossing singularities). In fact, conjecturally (\cite{AK2000}, Conjecture 0.2) after a finite base change and a birational modification, the family can always be brought to a semistable family (\cite{AK2000}). Hence it's natural to ask what kind of Fano variety with semistable singularities appear in a semistable family with Fano general fibers.
  \item Let $X$ be a projective smooth Fano variety over a local field $K$ of mixed characteristic. One can not hope generally that $X$ has a smooth model even after a base field extension. However, it is conjectured that after a finite extension $K\subseteq L$, $X_L$ would have a semistable model, i.e., there exists a semistable family $\sX$ over $O_L$ whose generic fiber is isomorphic to $X_L$. Therefore it is natural to ask what kind of Fano variety with semistable singularities appear to be the special fiber of a semistable model of a Fano manifold defined over $K$.
\end{enumerate}
The degeneration of Fano manifolds is studied by T. Fujita (\cite{Fujita1990}), who gives a complete list of reducible singular fibers in a (minimal) family of del Pezzo manifolds over an algebraic curve (A Fano manifold $X$ is called del Pezzo if $\textrm{ind}X=\dim X-1$). Latter, Y. Kachi \cite{Kachi2007} proves that all d-semistable normal crossing del Pezzo surfaces are contained in Fujita's list by showing that each d-semistable del Pezzo surface has a smoothing. This smoothable property is generalized to arbitrary dimensional d-semistable normal crossing Fano varieties defined over an algebraically closed field of characteristic zero by Tziolas \cite{Tziolas2015}. Since any family over a curve can always be modified into a family with normal crossing fibers (\cite{Mumford1973}), Tziolas's work classifies the simplest Fano degenerate fibers of a 1-parameter family of Fano manifolds.

However the Mori fibrations generally have higher dimensional bases, in which case the normal crossing singularities are no longer enough for the degenerate fibers (\cite{Karu1999}). The simplest example in which the degeneration can not be normal crossing is the 2-parameter family of surfaces defined by $t_1=x_1y_1$ and  $t_2=x_2y_2$. Conjecturally (\cite{AK2000}, Conjecture 0.2), the best singularities that one can hope for is the semistable singularities (\'etale locally a product of normal crossing singularities) and any geometric generic integral family can be modified into a semistable family. Given pairs $(X,D_X)$ and $(Y,D_Y)$ such that $X$, $Y$ are smooth and $D_X$, $D_Y$ are normal crossing divisors. A log morphism $f:(X,D_X)\rightarrow(Y,D_Y)$ (A morphism $f:X\rightarrow Y$ such that $f^{-1}(D_Y)\subseteq D_X$) is called semistable if formal locally $f$ is isomorphic to the spectrum of
$$k[[y_1,\cdots,y_r]]\rightarrow k[[x_1,\cdots,x_n]]$$
$$y_i\mapsto x_{l_{i-1}}+\cdots x_{l_{i}}$$
where $0=l_0\leq l_1<\cdots<l_r=n$. $D_X$ (resp. $D_Y$) are defined by the product of some of $x_i$ (resp. $y_i$).

In this paper, we study the the problem that which kind of semistable Fano varieties may appear in a semistable family whose general fibers are Fano manifolds. We work over an \textit{arbitrary} field and allow the degenerate fiber to have self intersections. The main results are:
\begin{thm}[Corollary \ref{cor_smoothing_0}]\label{thm_main1}
Let $k$ be a field of characteristic 0 and $X$ be a Fano semistable variety. Let $r\geq 1$ be an integer. Then the followings are equivalent:
\begin{enumerate}
  \item there exists a smooth variety with a normal crossing divisor $(\sX,D_{\sX})$ which is semistable (Definition \ref{defn_semistable}) over $(B,D_B)$ such that
\begin{enumerate}
  \item $B$ is an $r$ dimensional smooth variety over $k$, $0\in B$ is a $k$-point, and $D_B$ is a simple normal crossing divisor whose number of branches at $0$ is $r$;
  \item $\sX_0\simeq X$.
\end{enumerate}
  \item $X$ has a log structure (in the sense of Kato-Fontain-Illusie) of semistable type over $(k,\mathbb{N}^r\mapsto 0)$.
\end{enumerate}
\end{thm}
A variety with semistable singularities is called 'Fano' if its dualizing sheaf is an anti-ample invertible sheaf. For the readers who are not familiar with log geometry in the sense of Fontaine-Illusie-Kato, the condition that a semistable variety has a log structure of semistable type is a global condition on how do the components of the variety intersect each other. For example, if $X$ is a union of two smooth components $X_1$ and $X_2$ which intersect transversely along a smooth variety $D$. Then $X$ has a log structure of semistable type if and only if $N_{X_1 /D}\otimes N_{X_2/D}$ is a trivial line bundle on $D$.
\begin{thm}[Corollary \ref{cor_smoothing_p}]\label{thm_main2}
Let $k$ be a field of characteristic $p>0$ and $X$ is a log variety semistable log smooth over $(k,\mathbb{N}^r\mapsto 0)$ for some $r\geq 0$. Assume $\dim X<p$. If $X$ is Fano and admits a log smooth lifting over $(C_2(k),\mathbb{N}^r\mapsto 0)$, then there exists a smooth variety with a normal crossing divisor $(\sX,D_{\sX})$ which is semistable over $(B,D_B)$ such that
\begin{enumerate}
  \item $B$ is an $r$ dimensional smooth variety over $k$, $0\in B$ is a $k$-point, and $D_B$ is a simple normal crossing divisor whose number of branches at $0$ is $r$;
  \item $\sX_0\simeq X$ as log varieties.
\end{enumerate}
Moreover if $r=1$ (i.e., $X$ is $d$-semistable in the sense of \cite{Friedman1983}), then $X$ appears in a semistable reduction over $C(k)$.
\end{thm}
The technique that we use in this paper is the log deformation, which is first used by Y. Kawamata and Y. Namikawa to smooth certain normal crossing Calabi-Yau varieties in \cite{Kawamata1994} (although they do not use the formal language of log geometry). There are two advantages by using deformations to respect log structures:
\begin{enumerate}
  \item If we endow a normal crossing singularity $x_1\cdots x_r=0$ with the canonical log structure, it becomes smooth in the log scheme category. Although there are many types of deformations of $x_1\cdots x_r=0$, the log smooth deformation of $x_1\cdots x_r=0$ which respects the canonical log structure is quite simple. There are two types of log smooth deformation of $x_1\cdots x_r=0$, one keep the type of the singularity ($x_1\cdots x_r=0$), the other one smooth it ($x_1\cdots x_r=t$). Therefore, if we choose the log smooth deformation which smooths the singularities, we automatically get a semistable family.
  \item Generally, semistable Fano varieties are obstructed. However one can show that they are unobstructed under the log smooth deformation. In fact their log obstruction space vanish (Proposition \ref{cor_vanishing_0}).
\end{enumerate}
In \cite{Tziolas2015}, Tziolas proves the vanishing of the log obstruction space of a $d$-semistable normal crossing Fano variety by normalizing the singularities and reducing the vanishing theorem to the Akizuki-Nakano-Kodaira vanishing theorem of log pairs. However, in the semistable case, the method of normalizing becomes combinationally more complicated. On the other hand, the method of normalizing is not so effective in positive characteristic. In this paper, we use the full power of log geometry which is introduced by Fontaine-Illusie and is developed by K. Kato \cite{KKato1988}. By using K. Kato's decomposition theorem of log de Rham complex (Theorem \ref{thm_Kato}), we are able to prove a general Akizuki-Nakano-Kodaira type vanishing theorem for semistable log varieties (Theorem \ref{thm_vanishing_p} and Corollary \ref{cor_vanishing_0}) \textit{over any field}. The vanishing of log obstruction space of semistable log Fano varieties is an easy consequence.

As long as the log deformation is unobstructed, we are able to lift the semistable log variety over a complete ring (provided a lifting of an ample line bundle). Then we prove the limit preserving property of semistable log smooth morphisms (Proposition \ref{prop_limit_preserving}) and use Artin's approximation theorem to extend the family to a variety base.

This paper is organized as follows:

In section 2 we introduce the notions in log geometry which are necessary for the paper. We also introduce K. Kato's obstruction theory of log smooth deformations.

In section 3 we proves an Akizuki-Nakano-Kodaira type vanishing theorem for semistable log varieties (Theorem \ref{thm_vanishing_p} and Corollary \ref{cor_vanishing_0}) by using K. Kato's decomposition theorem of log de Rham complex (Theorem \ref{thm_Kato}). As a consequence, we prove that the obstruction space of the log smooth deformation of a semistable log Fano variety vanishes.

In section 4 we prove the limit preserving property of the semistable log smooth morphisms (Proposition \ref{prop_limit_preserving}) and prove the main results of this paper (Theorem \ref{thm_main1} and Theorem \ref{thm_main2}).

\textbf{Notations:}

We mainly follow the notions and notations in \cite{KKato1988} with some exceptions:
\begin{itemize}
  \item We use the capital letters $X$, $Y$, etc. to denote log schemes. If $X$ is a log scheme, we denote $\underline{X}$ be the underlying scheme and $\alpha_X:\sM_X\rightarrow\sO_X$ be the log structure.
  \item We denote a log cotangent sheaf by $\Omega$ instead of $\omega$ as in \cite{KKato1988}.
\end{itemize}
\section{Logarithmic Geometry and Logarithmic Deformation}
A log scheme is a triple $X=(\underline{X},\sM_X,\alpha)$ consists of
\begin{itemize}
  \item a scheme $\underline{X}$,
  \item a sheaf $\sM_X$ (on the \'etale topology on $X$) of monoids and
  \item a morphism of sheaf of monoids $\alpha:\sM_X\rightarrow\sO_X$ from $\sM_X$ to the multiplication monoid $(\sO_X,\times)$ such that $$\alpha|_{\alpha^{-1}\sO_X^\ast}:\alpha^{-1}\sO_X^\ast\rightarrow\sO_X^\ast$$ is an isomorphism.
\end{itemize}
$(\sM_X,\alpha)$ is called the log structure of the log scheme. Usually, $\alpha$ is omitted in the notation if there is no danger of ambiguity.

For any morphism of sheaf of monoids $\alpha:\sP\rightarrow\sO_X$, one can associate to it a log structure $\sP^a\rightarrow\sO_X$ functorially. If $\sP$ is a constant sheaf of monoids, then it is called a chart of $\sP^a$.

A typical example of log scheme is the log pair: let $(X,D)$ be a pair consists of a scheme $X$ and a reduced subscheme $D$ on $X$ of codimension 1. We can construct a log structure $\sM_D$ on $X$ by
$$\sM_D(U)=\{f\in\sO_X(U)|\textrm{supp}(f)\subseteq D\times_XU\}.$$

A morphism of log scheme $f:X\rightarrow Y$ is consist of a morphism of the underlying schemes $\underline{f}:\underline{X}\rightarrow\underline{Y}$ and a morphism of sheaf of monoids $f^\ast:f^{-1}\sM_Y\rightarrow\sM_X$. We say that $f$ is strict if $(f^{-1}\sM_Y)^a\simeq\sM_X$.
$f$ is called log smooth if
\begin{enumerate}
  \item the underlying morphism is locally of finite presentation and
  \item for any commutative diagram
$$
\xymatrix{
T'\ar[r] \ar[d]^i & X \ar[d]^f \\
T \ar[r] \ar@{.>}[ur]^g & Y
}
$$
of log schemes such that $i$ is a strict closed immersion whose ideal of definition $I$ satisfies $I^2=(0)$, there exists locally on $T$ a dotted arrow $g$ rendering the diagram commutative.
\end{enumerate}
A typical example of log smooth morphism is the semistable morphism:
\begin{defn}\label{defn_semistable}
Let $X$ and $Y$ be Noetherian schemes. Suppose that $D_X\subseteq X$ and $D_Y\subseteq Y$ are reduced Cartier divisors. A morphism of pairs $f:(X,D_X)\rightarrow (Y,D_Y)$ is called \textbf{semistable} if the following two conditions hold.
\begin{enumerate}
  \item $X$, $Y$ are regular and $D_X\subseteq X$, $D_Y\subseteq Y$ are normal crossing divisors,
  \item for each point $x\in X$ with $f(x)=y$, there are \'etale morphisms $U\rightarrow X$ and $V\rightarrow Y$ which send $x'\in U$ and $y'\in V$ to $x$ and $y$ respectively and a morphism $g$ rendering the diagram
$$\xymatrix{
U \ar[r]^{x'\mapsto x} \ar[d]^g & X \ar[d]^f \\
V \ar[r]^{y'\mapsto y} & Y
}$$
commutative. Here $g$ is formally isomorphic to the spectrum of
$$\widehat{\sO_{V,x}}\rightarrow\widehat{\sO_{V,x}}[[x_1,\cdots,x_n]]/I$$
$$f\mapsto f$$
where the ideal $I$ is generated by
$$r_i-\prod_{j=l_{i-1}+1}^{l_i}x_j.$$
Here $0=l_0<l_1\cdots<l_m\leq n$ and $r_i$ are the defining functions of the formal branches of $D_Y\times_YV$ in $V$.
\end{enumerate}

\end{defn}
A semistable morphism is log smooth if $X$ and $Y$ are endowed with the log structures induced by the divisors $D_X$ and $D_Y$. This is a direct consequence of the following criterion of K. Kato.
\begin{thm}\label{thm__log_smooth_chart}(\cite{KKato1988}, Theorem 3.5)
Let $f:X\rightarrow Y$ be a morphism of fine log schemes. Then $f$ is log smooth if and only if given any \'etale local chart $Q\rightarrow \sM_Y$, there exist an \'etale local chart $P\rightarrow \sM_X$, and a chart of $f$, $Q\rightarrow P$, such that
\begin{description}
  \item[(a)] $Ker(Q^{gp}\rightarrow P^{gp})$ and the torsion part of $Coker(Q^{gp}\rightarrow P^{gp})$ are finite groups of order invertible on $X$, and
  \item[(b)] the induced morphism $X\rightarrow Y\times_{Spec\mathbb{Z}[Q]}Spec\mathbb{Z}[P]$ is smooth (in the usual sense).
\end{description}
\end{thm}
Let $f:X\rightarrow Y$ be a semistable morphism. By restricting the log structures on the fibers of $f$ at $y\in Y$, we get a log variety $X_y$ which is log smooth over $(k,\mathbb{N}^r\mapsto 0)$ where $r$ is the number of formal branches of $D_Y$ at $y$. This suggests the following definition.
\begin{defn}\label{defn_semistable_log}
A log smooth variety $X$ over a field $(k,\mathbb{N}^r\mapsto 0)$ is of semi-stable type if for every point $x\in X$, there exist
\begin{itemize}
  \item a pointed scheme $(U,u)$,
  \item an \'etale morphism $U\rightarrow X$ which sends $u$ to $x$ (we do not require that $k(u)\simeq k(x)$) and
  \item a diagram of log schemes
$$\xymatrix{
U\ar[d]\ar[r]^-g & \spec(k[x_1,\cdots,x_n]/(x_1\cdots x_{l_1},x_{l_1+1}\cdots x_{l_2},\cdots,x_{l_{r-1}+1}\cdots x_n), \mathbb{N}^n) \ar[dl]^\pi\\
(k,\mathbb{N}^r\mapsto 0), & \\
}
$$
\end{itemize}
where $g$ is smooth and the log structure of $$\spec(k[x_1,\cdots,x_n]/(x_1\cdots x_{l_1},x_{l_1+1}\cdots x_{l_2},\cdots,x_{l_{r-1}+1}\cdots x_n)$$
is induced by $$\alpha:\mathbb{N}^n\rightarrow k[x_1,\cdots,x_n]/(x_1\cdots x_{l_1},x_{l_1+1}\cdots x_{l_2},\cdots,x_{l_{r-1}+1}\cdots x_n),\quad \alpha(e_i)=x_i.$$
Here $e_i=(0,\cdots, 1, \cdots, 0)\in \mathbb{N}^s$ where $1$ is placed on the $i$-th component.

Under a suitable order of $x_1,\cdots,x_n$ and $0=l_0\leq l_1<\cdots<l_r=n$, the map on the log structure of the morphism $\pi$ is induced by
$$\mathbb{N}^r\rightarrow \mathbb{N}^n,\quad e_i\mapsto e_{l_{i-1}+1}+\cdots e_{l_{i}}.$$
In this case we say that $X$ is \textbf{semistable log smooth} over $(k,\mathbb{N}^r\mapsto 0)$.
\end{defn}

Given a morphism of log schemes $f:X\rightarrow Y$, the relative cotangent sheaf of $f$ is defined by
$$\Omega_{X/Y}:=\Omega_{\underline{X}/\underline{Y}}\oplus(\sM_X^{gp}\otimes_{\mathbb{Z}}\sO_X)/\sim,$$
where $\sim$ is a $\sO_X$-submodule generated by
\begin{enumerate}
  \item $(d\alpha(a),0)-(0,\alpha(a)\otimes a)$ with $a\in\sM_X$, and
  \item $(0,1\otimes a)$ with $a\in \textrm{Image}(f^{-1}(\sM_Y)\rightarrow\sM_X)$.
\end{enumerate}
The canonical morphism $\sM^{gp}_X\rightarrow\Omega_{X/Y}$ is denoted by $\textrm{dlog}$. If $f$ is log smooth, then $\Omega_{X/Y}$ is locally free. The log cotangent sheaf controls the log smooth deformation of a log smooth variety (\cite{KKato1988}, \cite{FKato1996}).

\begin{thm}[\cite{KKato1988}, Proposition 3.14]\label{thm_log_deformation}
Let $f:X\rightarrow Y$ be a smooth morphism between fine log schemes and let $i:Y\rightarrow Y'$ be an strict closed immersion such that $Y$ is defined in $Y'$ by a square zero ideal $I$. Then there is an obstruction class $ob_{f,i}\in H^2(X,(\Omega_{X/Y})^\vee\otimes I)$ such that $ob_{f,i}=0$ if and only if there exists a log smooth morphism of fine log schemes $f':X'\rightarrow Y'$ whose restriction on $Y$ is isomorphic to $f$.
\end{thm}
\begin{rmk}\label{rmk_lifting}
Let $Y$ be a fine log scheme and  $i:Y\rightarrow Y'$ be a strict thickening (i.e., $i$ is a closed immersion defined by a nilpotent ideal sheaf where $i$ is strict as a log morphism). Let $X$ be a fine log scheme which is log smooth over $Y$. If $X$ is affine, then there exists a unique lifting (up to non-unique isomorphism) over $Y'$ (\cite{KKato1988}, Proposition 3.14). \'Etale locally a log smooth lifting can be described as follows (loc. cit.).

Choosing a chart $P$ of $X$ and a chart $Q$ of $Y$ as in Theorem \ref{thm__log_smooth_chart}, we get the following diagram
$$\xymatrix{
X \ar@{.>}[r] \ar[d]^f & X' \ar@{.>}[d]^{f'} \\
\spec(\bZ[P])\times_{\spec(\bZ[Q])}Y\ar[r] \ar[d] & \spec(\bZ[P])\times_{\spec(\bZ[Q])}Y' \ar[d] \\
k \ar[r] & Y'
},
$$
where $f$ is strict and \'etale. We can complete the diagram by the dotted arrows so that $f'$ is strict, \'etale and the square on the top is a fiber product (SGA I \cite{Grothendieck1971}, Expos\'e 1). Here $X'$ is a log smooth lifting of $X$ to $Y'$.
\end{rmk}

\section{Kodaira-Akizuki-Nakano Vanish for Semistable Log Varieties}
In this section we prove that semistable log Fano varieties are unobstructed under the log smooth deformations. To prove this we need to establish the Kodaira-Akizuki-Nakano vanishing theorem for semistable log varieties. As corollaries, we prove the main results in this paper.
\begin{defn}
Let $k$ be a field and $X$ be a variety over $k$. $X$ is said to be semistable if for each point $x\in X$, there exists an \'etale morphism $U\rightarrow X$ which sends $u\in U$ to $x$ (we do not require that $k(u)\simeq k(x)$) such that $$\widehat{\sO_{U,u}}\simeq k(u)[[x_1,\cdots,x_n]]/(\prod_{j=l_{i-1}+1}^{l_i}x_j).$$
Here $0=l_0<l_1\cdots<l_r\leq n$.
\end{defn}
This definition allows singularities like $x^2+y^2=0$ defined over $k=\mathbb{R}$.

Let $X$ be a semistable variety over $k$. Since it is \'etale locally a product of normal crossing singularities, $X$ is Gorenstein and the dualizing sheaf $\omega_{X/k}$ is an invertible sheaf.
\begin{defn}
A semistable variety $X$ is called Fano if its dualizing sheaf $\omega_{X/k}$ is anti-ample.
\end{defn}
For a semistable log variety, the log canonical sheaf is isomorphic to the dualizing sheaf of the underlying variety.
\begin{lem}\label{lem_log_canonical}
Let $X$ be a semistable log variety over $\mathbf{k}=(k,\mathbb{N}^r\mapsto 0)$, then $\bigwedge^{\dim X }\Omega_{X/\mathbf{k}}\simeq\omega_{X/k}$.
\end{lem}\label{thm_Kato}
\begin{proof}
For simplicity, let us consider the normal crossing singularity. Denote by $Z=\{x_1\cdots x_r=0\}$ a closed subspace of $U=\spec(k[x_1,\cdots,x_n])$. Then $\bigwedge^{n-r}\Omega_{Z/\mathbf{k}}$ is an invertible sheaf generated by
$$\frac{dx_1}{x_1}\wedge\cdots\wedge\frac{dx_r}{x_r}\wedge dx_{r+1}\wedge\cdots\wedge dx_n.$$
The dualizing sheaf $\omega_{Z/k}$ has the same description by the adjunction formula. Since semistable varieties are locally products of normal crossing singularities, the lemma is proved (We omit the compatible verification here).
\end{proof}
\begin{thm}\label{derham_split}(\cite{KKato1988} Theorem 4.12)
Let $k$ be a field of characteristic $p>0$ and let $X$ be a semistable log smooth variety over $\mathbf{k}=(k,\mathbb{N}^r\mapsto 0)$. If $X$ has a log smooth lifting over $(C_2(k),\mathbb{N}^r\mapsto 0)$, then there is an isomorphism
$$\tau_{<p}F_\ast\Omega^{\bullet}_{X/\mathbf{k}}\simeq\oplus_{0\leq i<p}\Omega^i_{X'/\mathbf{k}}[-i]$$
in $D(X')$. Here $F:X\rightarrow X'$ is the relative Frobenius over $\mathbf{k}$.
\end{thm}

\begin{thm}\label{thm_vanishing_p}
Let $k$ be a field of characteristic $p>0$ and let $X$ be a semistable log smooth variety over $\mathbf{k}=(k,\mathbb{N}^r\mapsto 0)$. Fix $\dim X<p$. Suppose that $X$ has a log smooth lifting over $(C_2(k),\mathbb{N}^r\mapsto 0)$. Then for any ample bundle $A$ on $X$, we have that:
$$H^i(X,\Omega^j_{X/\mathbf{k}}\otimes A)=0$$ for $\dim X<i+j<p$, $i>0$, and
$$H^i(X,\Omega^j_{X/\mathbf{k}}\otimes A^{-1})=0$$ for $i+j<inf\{p, \dim X\}$, $i<\dim X$.
\end{thm}
\begin{proof}
Consider the Frobenius square
$$\xymatrix{
X \ar[rrd]^{F_X} \ar[rd]^-F \ar[rdd]_f & & \\
 & X' \ar[r]^g \ar[d]^{f'} & X \ar[d]^f \\
 & \mathbf{k} \ar[r]^{F_{\mathbf{k}}}& \mathbf{k}.
}$$
\textbf{First Part.} Since $A$ is ample, we know that $$H^i(\Omega^j_{X/\mathbf{k}}\otimes A^{p^{m+1}})=0,\quad m\gg 0$$ for $\dim X<i+j<p$, $i>0$. From the above vanishing result, $$h^i(Rf_\ast(\Omega^{\bullet}_{X/Y}\otimes A^{p^{m+1}}))=0$$ for $i>\dim X$. Due to Theorem \ref{derham_split}, we have that
$$\tau_{<p}F_\ast(\Omega^{\bullet}_{X/\mathbf{k}}\otimes A^{p^{m+1}})\simeq\bigoplus_{0\leq i<p}\Omega^i_{X'/\mathbf{k}}\otimes g^\ast A^{p^m}[-i].$$
Hence for $\dim X<i<p$, we obtain that $$H^i(\Omega^i_{X'/\mathbf{k}}\otimes g^\ast A^{p^m})=0.$$ Since the underlying morphism of $g$ is an isomorphism, $$H^i(\Omega^j_{X/\mathbf{k}}\otimes A^{p^m})=0$$ for $\dim X<i+j<p$, $i>0$. Continuing this induction, we get that
$$H^i(X,\Omega^i_{X/\mathbf{k}}\otimes A)=0$$ for $\dim X<i+j<p$, $i>0$.

\textbf{Second Part.} The proof of the second part is almost the same. Since $X$ is Gorenstein, by Serre-Grothendieck duality (cf.\cite{Hartshorne1966}), we have that $$H^i(X,\Omega^j_{X/\mathbf{k}}\otimes A^{-p^m})=0,\quad m\gg 0,\quad i+j<\dim X.$$ This can be proved as follows:

Since $X$ is Gorenstein, by Serre-Grothendieck duality (cf.\cite{Hartshorne1966}), we have that $$H^i(X,\Omega^j_{X/\mathbf{k}}\otimes A^{-p^m})^{\vee}=H^{\dim X-i}(X,\omega_{X/k}\otimes(\Omega^j_{X/\mathbf{k}})^{\vee}\otimes A^{p^m})=0,\quad m\gg 0$$ for $0<i<\dim X$. Hence $$H^i(\Omega^{\bullet}_{X/\mathbf{k}}\otimes A^{-p^m})=0,\quad 0<i<\dim X,\quad m\gg 0.$$
Due to Theorem \ref{derham_split}, we know that
$$\tau_{<p}F_\ast(\Omega^{\bullet}_{X/\mathbf{k}}\otimes A^{-p^{m}})\simeq\bigoplus_{0\leq i<p}\Omega^i_{X'/\mathbf{k}}\otimes g^\ast A^{-p^{m-1}}[-i].$$
We have $$H^i(X,\Omega^j_{X/\mathbf{k}}\otimes A^{-p^{m-1}})=0,\quad i+j<\textrm{min}(p,\dim X),\quad i<\dim X.$$
By induction we get that
$$H^i(X,\Omega^j_{X/\mathbf{k}}\otimes A^{-1})=0$$ for $i+j<\textrm{min}\{p, \dim X\}$, $i<\dim X$.
\end{proof}
In this case, $\bigwedge^{\dim X }\Omega_{X/\mathbf{k}}$ is isomorphic to the dualizing sheaf $\omega_{X/k}$. Therefore the two vanishing results are related by Serre-Grothendieck duality. We present here the separated proofs because our argument hold for more general log smooth varieties (e.g. allowing divisors), for which the dualizing sheaf is different from the log canonical sheaf.

By the standard reduction mod $p$ technique, we have the following corollaries.
\begin{cor}\label{cor_vanishing_0}
Let $k$ be a field of characteristic $0$ and let $X$ be a semistable log smooth variety over $\mathbf{k}=(k,\mathbb{N}^r\mapsto 0)$. Then for any ample bundle $A$ on $X$, we have that
$$H^i(X,\Omega^j_{X/\mathbf{k}}\otimes A)=0$$ for $\dim X<i+j$, $i>0$, and
$$H^i(X,\Omega^j_{X/\mathbf{k}}\otimes A^{-1})=0$$  for $i+j<\dim X$, $i<\dim X$.
\end{cor}
\begin{prop}\label{prop_vanish_ob}
Let $k$ be a field and $X$ be a semistable log Fano variety over $\mathbf{k}=(k,\mathbb{N}^r\mapsto 0)$, then
\begin{enumerate}
  \item if $\textrm{char}k=0$, then $H^2(X,(\Omega_{X/\mathbf{k}})^{\vee})=0$ and $H^2(X,\sO_X)=0$.
  \item if $\textrm{char}k=p>0$, $\dim X>p$ and $X$ admits a log smooth lifting over $(C_2(k),\mathbb{N}^r\mapsto 0)$, then $H^2(X,(\Omega_{X/\mathbf{k}})^{\vee})=0$ and $H^2(X,\sO_X)=0$.
\end{enumerate}
\end{prop}
\begin{proof}
By Theorem \ref{thm_vanishing_p} and Corollary \ref{cor_vanishing_0}, we see that
$$H^2(X,(\Omega_{X/\mathbf{k}})^{\vee})=H^{n-2}(X,\Omega_{X/\mathbf{k}}\otimes\omega_{X/k})=0.$$
By Lemma \ref{lem_log_canonical}, we obtain that
$$H^2(X,\sO_X)=H^2(X,\bigwedge^{\dim X}\Omega_{X/\mathbf{k}}\otimes\omega_{X/k}^{-1})=0.$$
\end{proof}

\section{Smoothing of Semistable Log Fano Varieties}
The vanishing theorems in the last section ensure the lifting of a semistable log Fano variety to a complete DVR. To extend this family to a variety base, we need to study the limit preserving property of semistable log smooth morphisms (Proposition \ref{prop_limit_preserving}). First we would like to generalize the notion of semistable log morphisms.
\begin{defn}
Let $Y$ be a log scheme which has a global chart $\bN^r$. A morphism $f:X\rightarrow Y$ is semistable log morphism if \'etale locally there exists a chart $\bN^n\rightarrow \sM_X$ on $X$ and a chart of $f$,
$$\bN^r=\bigoplus_{i=1}^r\bN e_i\rightarrow \bN^n=\bigoplus_{i=1}^n\bN e'_i$$
$$e_i\mapsto e_{l_{i-1}}+\cdots e_{l_{i}}$$
where $0=l_0\leq l_1<\cdots<l_r=n$, such that the induced morphism
$$X\rightarrow S\times_{\spec\mathbb{Z}[\bN^r]}\spec\mathbb{Z}[\bN^n]$$
is smooth in the usual sense.
\end{defn}
This definition generalizes Definition \ref{defn_semistable_log}. If $(Y,D_Y)$ is a smooth variety with the log structure induced by a simple normal crossing divisor $D_Y$, then the underlying morphism of a semistable log smooth morphism $f:X\rightarrow Y$ is semistable in the sense of Definition \ref{defn_semistable}.
\begin{prop}\label{prop_limit_preserving}
Given an $k[\bN^r]$-algebra $A$, denote $A^{\textrm{log}}$ be the log algebra $A$ whose log structure is given by the composition of $\bN^r\rightarrow k[\bN^r]$ and the structure homomorphism $k[\bN^r]\rightarrow A$. Then the pseudo-functor from the category of $k[\bN^r]$-algebras to the category of groupoids
$$\sF(k[\bN^r]\rightarrow A)=\textrm{ the groupoid of log scheme $X$ semistable log smooth over }A^{\textrm{log}}$$
is limit preserving, i.e., for any directed inverse system of $k[\bN^r]$-algebras $A_i$, the canonical functor
$$\sigma:\sF(\lim_iA_i)\rightarrow\lim_i\sF(A_i)$$
is an equivalence.
\end{prop}
\begin{proof}
Denote by $A$ the limit algebra $\lim_iA_i$.
The fully faithfullness of $\sigma$ is formal. It suffices to prove that $\sigma$ is essentially surjective, i.e., for any $X\in\sF(A)$, there exists an index $i_0$ and $X_{i_0}\in\sF(A_{i_0})$ such that its base change on $A$ is $X$.

Since $\underline{X}$ is locally of finite presentation over $\spec(A)$, there exists an index $i_1$ and a scheme $\underline{X_{i_1}}$ locally of finite presentation over $\spec(A_{i_1})$ such that its base change on $\spec(A)$ is $\underline{X}$. To construct the log structure, we consider the pseudo-functor
$$\textrm{Log}_{A_{i_1}}(X)=\textrm{ the groupoid of fine log structures on } X \textrm{ over } A^{\textrm{log}}_{i_1}.$$
It is proved in \cite{Olsson2003} that this pseudo-functor is represented by an algebraic stack locally of finite presentation over $A_{i_1}$. Therefore $\textrm{Log}_{A_{i_1}}$ is limit preserving. Applying $\textrm{Log}_{A_{i_1}}$ on the directed inverse system $\{X_{i_1}\times_{A_{i_1}}A_i\}_{i\geq i_1}$, we see that there exists an index $i_2\geq i_1$ and a fine log structure over
$$\underline{X_{i_2}}=\underline{X_{i_1}}\times_{\spec A^{\textrm{log}}_{i_1}}\spec A^{\textrm{log}}_{i_2}$$
such that the base change of the log scheme $X_{i_2}$ over $A$ is $X$. So far, $X_{i_2}$ is not necessarily semistable log smooth over $A^{\log}_{i_2}$. We have to show that there exists an index $i_0\geq i_2$ such that $$X_{i_0}=X_{i_2}\times_{\spec A^{\log}_{i_2}}\spec A^{\log}_{i_0}$$
is semistable log smooth over $A^{\log}_{i_0}$. Fixing a point $x\in X$, there exists an \'etale local chart $\bN^m\rightarrow \sM_X$ at $x$, and a chart of $f$,
$$\bN^r=\bigoplus_{i=1}^r\bN e_i\rightarrow \bN^n=\bigoplus_{i=1}^n\bN e'_i$$
$$e_i\mapsto e_{l_{i-1}+1}+\cdots e_{l_{i}}$$
where $0=l_0\leq l_1<\cdots<l_k=n$, such that the induced morphism
$$X\rightarrow \spec A^{\log}\times_{\spec\mathbb{Z}[\bN^r]}\spec\mathbb{Z}[\bN^n]$$
is smooth in the usual sense. There exists an index $i_3\geq i_2$ such that the diagram
$$\xymatrix{
X \ar[r] \ar[d] & \spec\mathbb{Z}[\bN^n] \ar[d]\\
\spec A^{\log} \ar[r] & \spec\mathbb{Z}[\bN^r]\\
}$$
factors through
$$\xymatrix{
X_{i_3} \ar[r] \ar[d] & \spec\mathbb{Z}[\bN^n] \ar[d]\\
\spec A^{\log}_{i_3} \ar[r] & \spec\mathbb{Z}[\bN^r]\\
}$$
and the induced morphism
$$X_{i_3}\rightarrow \spec A^{\log}_{i_3} \times_{\spec\mathbb{Z}[\bN^r]}\spec\mathbb{Z}[\bN^n]$$
is smooth in the usual sense. Hence locally at the pre-images of $x\in X$, $X_{i_3}$ is semistable log smooth over $A^{\log}_{i_3}$. Since $X$ is quasi-compact, there exists an index $i_0\geq i_2$ such that $X_{i_0}=X_{i_2}\times_{A^{\log}_{i_2}}A^{\log}_{i_0}$ is semistable log smooth over $A^{\log}_{i_0}$.
\end{proof}
\begin{thm}\label{thm_semistable_lifting}
Let $k$ be a field and $X$ be a semistable projective log variety over $\mathbf{k}=(k,\mathbb{N}^r\mapsto 0)$ such that $H^2(X,(\Omega_{X/\mathbf{k}})^{\vee})=0$ and $H^2(X,\sO_X)=0$. Then there exists a log variety $\sX$ which is semistable log smooth over $(B,D_B)$ such that
\begin{enumerate}
  \item $B$ is a $r$ dimensional smooth variety over $k$, $0\in B$ is a $k$-point, and $D_B$ is a simple normal crossing divisor whose number of branches at $0$ is $r$;
  \item $\sX_0\simeq X$ as log varieties.
\end{enumerate}
\end{thm}
\begin{proof}
By Theorem \ref{thm_log_deformation}, $X$ has a log smooth lifting $\widehat{\sX}$ over the formal log scheme $$\widehat{B}=(\textrm{Spf}k[[x_1,\cdots,x_r]],\mathbb{N}^r=\oplus_{i=1}^r\mathbb{N}e_i, e_i\mapsto x_i).$$
Since $H^2(X,\sO_X)=0$, any ample line bundle on $X$ has a lifting over $\widehat{B}$. Denote
$$B_1=(\spec k[[x_1,\cdots,x_r]],\mathbb{N}^r=\oplus_{i=1}^r\mathbb{N}e_i, e_i\mapsto x_i).$$
By Grothendieck's existence theorem, there exists a log variety $\sX_1$ log smooth over $B_1$ whose formal completion is $\widehat{\sX}$.
By Remark \ref{rmk_lifting}, \'etale locally there is a commutative diagram
$$\xymatrix{
X \ar[r] \ar[d]^f & \sX \ar[d]^{f'} \\
\spec(\bZ[\bN^n])\times_{\spec(\bZ[\bN^r])}k \ar[r] \ar[d] &
\spec(\bZ[\bN^n])\times_{\spec(\bZ[\bN^r])}\spec k[[x_1,\cdots,x_r]] \ar[d] \\
k \ar[r] & \spec k[[x_1,\cdots,x_r]]\\
},
$$
where $f$ is strict and \'etale, $f'$ is strict, \'etale and the square on the top square is a fiber product. The morphism $\spec(\bZ[\bN^r])\rightarrow\spec(\bZ[\bN^n])$ is induced by the homomorphism of monoids
$$\bN^r=\bigoplus_{i=1}^r\bN e_i\rightarrow \bN^n=\bigoplus_{i=1}^n\bN e'_i$$
$$e_i\mapsto e_{l_{i-1}+1}+\cdots e_{l_{i}}.$$
Here $0=l_0\leq l_1<\cdots<l_r\leq n$. This shows that $\sX_1\rightarrow B_1$ is semistable.

To extend the family $\sX_1$ over a variety base, let us consider the pseudo-functor from the $k[\bN^r]$-algebras to groupoids:
$$\sG(k[\bN^r]\rightarrow A)=\{\textrm{groupoids of semistable log smooth varieties over }A^{\textrm{log}}\}$$
where $A^{\textrm{log}}$ is the log algebra $A$ whose log structure is given by the composition of $\bN^r\rightarrow k[\bN^r]$ and the structure homomorphism $k[\bN^r]\rightarrow A$. By Proposition \ref{prop_limit_preserving}, this functor is limit preserving. Hence by Artin's approximation (\cite{Artin1969}), there exists $\sX_1^h\in\sG(k[\bN^r]^h)$ whose base change on $k$ is $X$. Here $k[\bN^r]^h$ denotes the Henselization of $k[\bN^r]$ at the maximal ideal generated by $\bN^r$. Again by Proposition \ref{prop_vanish_ob}, $\sX_1^h$ extends to an $k[\bN^r]$-algebra $R$ which is \'etale over $k[\bN^r]$. Denote $B=\spec R^{\log}$, then this gives a semistable log smooth morphism $\sX\rightarrow B)$ which is required in the theorem.
\end{proof}
\begin{cor}\label{cor_smoothing_0}
Let $k$ be a field of characteristic 0 and $X$ be a Fano semistable variety. Let $r\geq 1$ be an integer. Then the followings are equivalent:
\begin{enumerate}
  \item there exists a log variety $\sX$ which is semistable log smooth over $(B,D_B)$ such that
\begin{enumerate}
  \item $B$ is an $r$ dimensional smooth variety over $k$, $0\in B$ is a $k$-point, and $D_B$ is a simple normal crossing divisor whose number of branches at $0$ is $r$;
  \item $\sX_0\simeq X$ as log varieties.
\end{enumerate}
  \item $X$ has a log structure of semistable type over $(k,\mathbb{N}^r\mapsto 0)$.
\end{enumerate}
\end{cor}
\begin{proof}
A semistable morphism $f:(\sX,D_{\sX})\rightarrow (B,D_B)$ is log smooth if we endow $\sX$ and $B$ with the log structures induced by the divisors $D_{\sX}$ and $D_B$. If $X$ is isomorphic to a fiber $\sX_b$ of a semistable morphism, then the log structure restricted on $\sX_b$ gives $X$ a log structure of semistable type over $(k,\mathbb{N}^r\mapsto 0)$ where $r$ is the number of formal branches of $D_B$ passing $b$.

The converse is the combination of Proposition \ref{prop_vanish_ob} and Theorem \ref{thm_semistable_lifting}. Since $\omega_{\sX/B'}$ is locally free, there exists an open neighbourhood $U$ of 0 such that $\omega_{\sX_b}$ is anti-ample for every $b\in U$.
\end{proof}

By the same argument, we get the following corolary.
\begin{cor}\label{cor_smoothing_p}
Let $k$ be a field of characteristic $p>0$ and $X$ be a log variety $X$ semistable log smooth over $(k,\mathbb{N}^r\mapsto 0)$ for some $r\geq 0$. Fix $\dim X<p$. If $X$ is Fano and admits a log smooth lifting over $(C_2(k),\mathbb{N}^r\mapsto 0)$, then there exists a smooth variety with a normal crossing divisor $(\sX,D_{\sX})$ which is semistable over $(B,D_B)$ such that
\begin{enumerate}
  \item $B$ is an $r$ dimensional smooth variety over $k$, $0\in B$ is a $k$-point, and $D_B$ is a simple normal crossing divisor whose number of branches at $0$ is $r$;
  \item $\sX_0\simeq X$ as log varieties.
\end{enumerate}
In particular, if $X$ is semistable log smooth over $(k,\mathbb{N}\mapsto 0)$ (i.e., $d$-semistable in \cite{Friedman1983}), then $X$ appears in a semistable reduction over $C(k)$.
\end{cor}

\end{document}